\documentclass{article}
\usepackage{amsmath}
\usepackage{amsthm}
\usepackage{amssymb}
\usepackage{verbatim}
\usepackage{graphicx}
\usepackage{float}
\usepackage{geometry}
\usepackage{indentfirst}
\usepackage{tikz-cd}

\newtheorem{theorem}{Theorem}
\newtheorem{lemma}{Lemma}
\newtheorem{proposition}{Proposition}

\title{Gradient Growth for Vlasov-Poisson in background charge distributed by Desingularized Dirac Delta}
\author{Sangwook Tae}

\numberwithin{equation}{section}
\numberwithin{theorem}{section}
\numberwithin{lemma}{section}
\numberwithin{proposition}{section}

\begin{document}
	\maketitle
	\begin{abstract}
		We prove the superlinear gradient growth for the one dimensional Vlasov-Poisson equation on the torus. We show them by first constructing the stationary solution which has a hyperbolic flow (locally). Then we perturb the solution and use the $L^\infty$ stability of the characteristic velocity field to show the superlinear growth. This proof was inspired by the work of Denisov \cite{growth_EE}.
	\end{abstract}
	\section{Introduction}
	\subsection{Main Results}
	Let us consider the one dimensional Vlasov-Poisson equation on the Torus $\mathbb{T}=\mathbb{R}/\mathbb{Z}$:
	\begin{flalign}
		\label{Vlasov}
		\begin{split}
			&\partial_{t}f(x, v, t)+v\cdot\partial_{x}f(x, v, t)-E(x,t)\cdot\partial_{v}f(x, v, t)=0\\
			&E(x, t)=\int_{0}^{1}K(x, y)\left[\rho(y)-\int_{-\infty}^{\infty}f(y, v, t)dv\right]dy.
		\end{split}
	\end{flalign}
	where the kernel $K$ is defined as
	\begin{flalign*}
		K(x, y)=\begin{cases}
			y, & \text{for } 0\leq x< y\\
			y-1, & \text{for } y< x\leq 1
		\end{cases}
	\end{flalign*}
	Here, $f$ is the distribution function for the electron in spacetime and $\rho$. is the background ion density. We assume $\int_{\mathbb{T}\times\mathbb{R}} f(t, x, v)dvdx=1=\int_{\mathbb{T}}\rho(y)dy$ so that the total charge is 0 on $\mathbb{T}$. This condition is needed to avoid contradictions. \\
	
	Equation (\ref{Vlasov}) describes the evolution of the distribution function for electrons under the electric field generated by themselves and background ion \cite{Majda}. The electric field obeys the one dimensional Poisson equation, which implies that the electric field satisfies the formula in the second line of (\ref{Vlasov}).  It can be used to study collisionless plasmas with heavy ions so that the positions of the ions remain fixed. \\
	
	In this paper, we are concerned for the growth of the gradient growth for the distribution function $f$. This topic is directly related to the small scales creations, due to the fact that $f$ is transported along the characteristics, which implies the conservation of $L^p$ norms. There are many results for the gradient growth for the vorticity in two dimensional Euler equations, but it seems that there are not so many results for the Vlasov-Poisson equation. Our result establishes the certain lower bound for the rate of growth for $\Vert\nabla_{x,v} f\Vert_{L^\infty}$ in time.
	
	We state the main theorem as follows:
	\begin{theorem}
		Consider the one dimensional Vlasov Poisson equation (\ref{Vlasov}) with the following density profile for background ion: 
		\begin{flalign*}
			\rho(y)=\rho_{\epsilon}(y)=\begin{cases}
				\frac{1}{\epsilon} & x\in[0, \frac{\epsilon}{2})\cup (1-\frac{\epsilon}{2}, 1]\\
				0 & \text{otherwise}
			\end{cases}
		\end{flalign*}
		which is desingularized version of Dirac delta distribution.
		Then, for any sufficiently small $\epsilon>0$, there exists a smooth initial data $f_0(x, v)$ such that the corresponding solution $f$ exhibits the following superlinear gradient growth in time:
		\begin{flalign*}
			\lim_{T\to\infty}\frac{1}{T^2}\int_0^T \Vert \nabla f(t)\Vert_{L^\infty} dt=\infty.
		\end{flalign*}
	\end{theorem}
	\subsection{Liturature}
	There are several results for gradient growth results for certain systems. For example, in the case of 2-dimensional incompressible Euler equation on the torus $\mathbb{T}^2$, Denisov \cite{growth_EE} proved the superlinear growth for $\Vert \nabla\omega(t)\Vert_{L^\infty}$ in time, where $\omega$ is the vorticity. Zlatos \cite{grotwh_EE_exp} obtained the exponential growth for such quantity for $C^{1, \alpha}$ some initial data with $C^{1, \alpha}$ vorticity.  Kiselev and Sverak \cite{grwoth_EE_bd} found the initial condition with the double exponential gradient growth in the case of the presence of the boundary.\\
	
	For the case of Vlasov-Poisson equation, the existence of solution with linear gradient growth can be obtained via Landau damping. See \cite{Gagnebin}, \cite{Ionescu}, and \cite{Mouhot} for detailed discussion about Landau damping. But as far as we know, we couldn't find any results beyond linear growth. \\
	
	We would also like to mention about the results for the steady state solutions in Vlasov-Poisson. Many result is based on constructing the conserved quantity which is continuous and positive definite so-called the Energy-Casimir method \cite{Casimir}. Let us mention about such examples. First, \cite{Yan&Gerhard_2}, Yan Guo and Gerhard Rein showed the existence of rotationally symmetric steady states stable under perturbations with the same symmetry. Gerhard Rein improved this result by removing the symmetric assumptions on perturbations in \cite{G.R}. Also, in \cite{G.W}, G.Wolansky constructed symmetric stable steady states for three-dimensional Vlasov-Poisson system in the stellar case. In \cite{Y.G}, Yan Guo constructed so-called ``Camm" type stable steady states in the same situation. More results concerning the stellar case can be found in \cite{Jack}, \cite{Yan&Gerhard}, and \cite{Yieh-Hei}. In the case of plasma, Peter Braasch, Gerhard Rein, and Jesenko Vukadinovich  constructed the stationary steady states in 3 dimensions  without the stationary ion approximations in \cite{Peter}. In \cite{Maria}, Maria J. Caceres, Jose A. Carrillo, and Jean  Dolbeault constructed $L^p$ stable states in the case of electrons moving in an external electric field. Finally, in \cite{Twisting}, Sangwook constructed the steady state for one dimensional periodic case, with background ion distributed according to the desingularized Dirac Delta distribution. This result is heavily used in our work.\\

	\subsection{Upper Bound for the Rate of Gradient Growth}
	Before we move on, we would like to remark some upper bound results for the gradient growth. The proof for the upper bound is standard, but to the best of our knowledge, there were no explicit comments of this fact. Thus, we introduce it here with proof.
	\begin{proposition}
		Consider the one-dimensional Vlasov-Poisson equation on the torus as given in (\ref{Vlasov}) and assume the consistency relation $\int_{\mathbb{T}\times\mathbb{R}} f(0, x, v)dvdx=1=\int_{\mathbb{T}}\rho(y)d$y. Assume that $\rho$ is bounded. Suppose that the $f(0, x, v)$ is smooth and compactly supported. Then, there exists a constant $C$ depending on the initial data $f(0, x, v)$ and the ion density $\rho$ such that the following inequality holds for any $t\geq0$:
		\begin{flalign*}
			\Vert \nabla_{x,v} f\Vert_{L^\infty}\leq C e^{Ct^2}.
		\end{flalign*}
	\end{proposition}
	\begin{proof}
		First, let us differentiate the equation (\ref{Vlasov}) with respect to $x$. Then, we have
		\begin{flalign*}
			\partial_{t}\partial_x f(x, v, t)+v\cdot\partial_{xx}f(x, v, t)-E(x,t)\cdot\partial_{v}\partial_xf(x, v, t)-\partial_xE(x,t)\cdot\partial_{v}f(x, v, t)=0.
		\end{flalign*}
		Using the Gauss's law, we have
		\begin{flalign*}
			\partial_{t}\partial_x f(x, v, t)+v\cdot\partial_{xx}f(x, v, t)-E(x,t)\cdot\partial_{v}\partial_xf(x, v, t)=\left(  \rho(x)-\int_{-\infty}^\infty f(t, x, u)du \right)\cdot\partial_{v}f(x, v, t).
		\end{flalign*}
		Since the left hand side is the transport operator of divergence-free vector field $(v, -E(x, t))$, we have the following estimate:
		\begin{flalign*}
			\frac{d}{dt}\Vert \partial_xf(t)\Vert_{L^\infty}\leq \left\Vert \left(  \rho(\cdot)-\int_{-\infty}^\infty f(t, \cdot, u)du \right)\cdot\partial_{v}f(t) \right\Vert_{L^\infty}\leq \left\Vert  \rho-\int_{-\infty}^\infty f(t, \cdot, u)du \right\Vert_{L^\infty}\Vert\partial_{v}f(t)\Vert_{L^\infty}.
		\end{flalign*}
		Now, we need to estimate $\left\Vert \int_{-\infty}^\infty f(t, \cdot, u)du \right\Vert_{L^\infty}$. Let us pick $Q>0$ such that $f_0(x, v)=0$ for $\vert v\vert \geq Q$.  Then, since $f$ is transported along the flow generated by $(v, -E(x, t))$,  we see that $f(t, x, v)=0$ for $\vert v\vert \geq Q+\int_{0}^t \Vert E(\tau)\Vert_{L^\infty}d\tau$. Since we have 
		\begin{flalign*}
			\vert E(x, t)\vert &\leq \int_0^1 \vert K(x,y)\vert \left\vert \rho(y)+\int_\mathbb{R} f(t, y, v)dv \right\vert dy \leq \int_0^1 \left(\rho(y)+\int_\mathbb{R} f(t, y, v)dv\right)dy=2.
		\end{flalign*}
		we see that  $f(t, x, v)=0$ for $\vert v\vert \geq Q+2t$. Therefore, we obtain the following estimate:
		\begin{flalign*}
			\frac{d}{dt}\Vert \partial_xf(t)\Vert_{L^\infty}&\leq \left\Vert  \rho-\int_{-\infty}^\infty f(t, \cdot, u)du \right\Vert_{L^\infty}\Vert\partial_{v}f(t)\Vert_{L^\infty}\leq \left( \Vert\rho\Vert_{L^\infty} + \int_{-Q-2t}^{Q+2t} \Vert f(t)\Vert_{L^\infty}du  \right)\Vert \partial_v f(t)\Vert_{L^\infty}\\&
			\leq \left( \Vert\rho\Vert_{L^\infty} + \int_{-Q-2t}^{Q+2t} \Vert f_0\Vert_{L^\infty}du  \right)\Vert \partial_v f(t)\Vert_{L^\infty}\leq C(1+t)\Vert \partial_vf(t)\Vert_{L^\infty}.
		\end{flalign*}
		Let us also  the equation (\ref{Vlasov}) with respect to $v$:
		\begin{flalign*}
			\partial_{t}\partial_v f(x, v, t)+v\cdot\partial_{x}\partial_vf(x, v, t)-E(x,t)\cdot\partial_{v}\partial_vf(x, v, t)+\partial_{x}f(x, v, t)=0,
		\end{flalign*}
		which gives
		\begin{flalign*}
			\frac{d}{dt}\Vert \partial_vf(t)\Vert_{L^\infty} \leq \Vert \partial_x f(t)\Vert_{L^\infty}.
		\end{flalign*}
		Adding these two inequality gives
		\begin{flalign*}
			\frac{d}{dt}\Vert \nabla_{x, v}f*(t)\Vert_{L^\infty}\leq C(1+t)\Vert \nabla_{x, v}f(t)\Vert_{L^\infty}.
		\end{flalign*}
		Using Gronwall inequality, we see that
		\begin{flalign*}
			\Vert \nabla_{x, v}f(t)\Vert_{L^\infty}\leq Ce^{Ct^2},
		\end{flalign*}
		and we are done.
	\end{proof}
	\subsection{Strategy for the Proof}
	The theorem is proven by first constructing the steady state solution which has a parabolic fixed point. This step is performed by converging the problem for the steady state solutions to a nonlinear algebraic equations and use the implicit function theorem. Second, we show the stability for the density in $L^1$, which implies the $L^\infty$ stability for the characteristic velocity field. Finally, we prove the theorem by perturbing the initial data and apply the argument done in \cite{growth_EE}. 
	\section{Proof for Superlinear Gradient Growth}
	\subsection{Preliminaries}
	\subsection{Characteristics and Stream Functions}
	\label{P}
	It is well known that the distribution function $f$ is conserved along the characteristic flow: i.e., the following holds.
	\begin{flalign}
		f(x, v, t)=f(\Phi_t^{-1}(x, v), 0)
	\end{flalign}
	where $\Phi$ is the flow corresponding to the following differential equations
	\begin{flalign}
		\frac{dx}{dt}=v\quad \text{and}\quad \frac{dv}{dt}=-E(x, t)
	\end{flalign}
	satisfying $\Phi_0=\operatorname{id}$.
	Here, $E(x, t)$ is the electric field. 
	This characteristic equation is a (non-autonomous) Hamiltonian system with Hamiltonian $$\psi(x, v, t)=\frac{1}{2}v^2+U(x, t)=\frac{1}{2}v^2+\int_{0}^{x}E(y, t)dy.$$
	As a consequence, this flow is (Lebesgue) measure-preserving.
	\subsection{Conservation of Energy}
	Another main ingredient in this paper is the conservation of energy. The energy of the state $f$ is defined as follows:
	$$H(f)=\int_{\mathbb{T}}\int_{\mathbb{R}} \frac{v^2}{2}f(x, v, t) dvdx+\int_{\mathbb{T}} \frac{1}{2} E(x, t)^2 dx.$$
	For sufficiently smooth solutions, it is known that this quantity is conserved in time:
	$$H(f(\cdot, \cdot, t))=H(f(\cdot, \cdot, 0)).$$
	\subsection{Construction of the reference steady state}
	We will use the steady state constructed in \cite{Twisting}. We briefly sketch how it is constructed. They first obtained solution for Dirac delta distribution and used implicit function theorem to get the solution in desingularized case. \\
	
	Let	
	\begin{flalign*}
		\rho_{\epsilon}(x)=
		\begin{cases}
			\frac{1}{\epsilon} & x\in[0, \frac{\epsilon}{2})\cup (1-\frac{\epsilon}{2}, 1]\\
			0 & \text{otherwise}.
		\end{cases}
	\end{flalign*}
	be the background charge distribution, which converges to the Dirac delta distribution centered at $x=0$ as $\epsilon\to0$.
	We choose the following ansatz as a steady state solution:
	$$f(x, v)=\varphi\left(\frac{v^2}{2}+U(x)\right)$$ 
	where $U(x)$ is the electric potential satisfying 
	$$E(x)=\frac{d}{dx}U(x),$$
	and $\varphi$ has the following form:
	\begin{flalign*}
		\varphi(E)=
		\begin{cases}
			(c-E)^{\frac{5}{2}} & E<c\\
			0 & \text{otherwise}
		\end{cases}
	\end{flalign*}
	where $c>0$ is to be determined.
	Then, by applying the Poisson equation for $U$, we see that $U$ must obey the following equation:
	\begin{flalign*}
		\frac{d^2}{dx^2}U(x)&=\rho_{\epsilon}(x)-\int_{-\infty}^{\infty}\varphi\left(\frac{v^2}{2}+U(x)\right)dv
	\end{flalign*}
	which, after calculation, can be rewritten in the following form:
	\begin{flalign*}
		\frac{d^2}{dx^2}U(x)&=\rho_{\epsilon}(x)-\int_{-\sqrt{2(c-U(x))}}^{\sqrt{2(c-U(x))}}\left(c-U(x)-\frac{v^2}{2}\right)^{\frac{5}{2}}dv=\rho_{\epsilon}(x)-\sqrt{2}\left( c-U(x) \right)^3\int_{-1}^{1} \left(1-u^2\right)^{\frac{5}{2}}du\\&=\rho_{\epsilon}(x)-\frac{5\sqrt{2}\pi}{16}(c-U(x))^3.
	\end{flalign*}
	We make the natural assumption $U(x)=U(1-x)$. As a consequence, we have $U'(0)=U'\left(\frac{1}{2}\right)=0$.
	Using this fact and integrating the equation several times, we obtain two equation that must be satisfied:
	\begin{flalign*}
		\frac{\epsilon}{2}=\int_{0}^{U(\frac{\epsilon}{2})}\frac{dy}{\sqrt{\frac{5\sqrt{2}\pi}{32}(-c^4+(c-y)^4)+\frac{2y}{\epsilon}}},
	\end{flalign*}
	and
	\begin{flalign*}
		\frac{1-\epsilon}{2}&=\int_{U(\frac{\epsilon}{2})}^{U(\frac{1}{2})}\frac{dy}{\sqrt{\frac{5\sqrt{2}\pi}{32}((c-y)^4-(c-U(\frac{1}{2}))^4)}}=\int_{U(\frac{\epsilon}{2})}^{c-\left( c^4-\frac{64}{5\sqrt{2}\pi \epsilon}U\left( \frac{\epsilon}{2} \right)  \right)^{\frac{1}{4}}}\frac{dy}{\sqrt{\frac{5\sqrt{2}\pi}{32}((c-y)^4-c^4)+\frac{2U(\frac{\epsilon}{2})}{\epsilon}}}.
	\end{flalign*}
	The detailed derivations are provided in \cite{Twisting}. By setting $h=\frac{U(\frac{\epsilon}{2})}{\epsilon}$, we can rewrite the two equations above in the following form:
	\begin{flalign*}
		\begin{cases}
			\frac{1}{2}=\int_{0}^{h} \frac{dz}{\sqrt{\frac{5\sqrt{2}\pi}{32}(-c^4+(c-\epsilon z)^4)+2z}}:=f_1(h, c, \epsilon)\\
			\frac{1}{2}=\frac{\epsilon}{2}+\int_{\epsilon h}^{c-\left(c^4-\frac{64h}{5\sqrt{2}\pi} \right)^{\frac{1}{4}}}\frac{dw}{\sqrt{\frac{5\sqrt{2}\pi}{32}((c-w)^4-c^4)+2h}}:=f_2(h, c, \epsilon).
		\end{cases}
	\end{flalign*}
	In \cite{Twisting}, it is shown that the solution for $(h, c)$ exists for every $\epsilon\in [0, \epsilon_0)$, where $\epsilon_0>0$ is sufficiently small. The strategy of the proof is to first obtain the solution for $\epsilon=0$ and use the implicit function theorem to extend to the case $\epsilon\in(0, \epsilon_0)$. This steady state will be used as a reference state. \\
	
	Before concluding this section, we show that this solution has a saddle point, which is an important ingredient for showing the gradient growth.
	\begin{proposition}
		\label{saddle}
		For $\epsilon\in(0, \epsilon_0)$, let $(h_\epsilon, c_\epsilon)$ denote the solution of the equation
		\begin{flalign*}
			\begin{cases}
				f_1(h, c, \epsilon)=\frac{1}{2}\\
				f_2(h, c, \epsilon)=\frac{1}{2}.
			\end{cases}
		\end{flalign*}
		and $U_\epsilon$ be the corresponding potential. Then, the stream function $\psi_\epsilon(x, v)=\frac{v^2}{2}+U_\epsilon(x)$ creates a saddle point at $(x, v)=(\frac{1}{2}, 0)$.
	\end{proposition} 
	\begin{proof}
		The gradient of the characteristic field is $$\nabla\nabla^\perp\psi(\epsilon)=\begin{pmatrix}
			0 & 1\\
			-\frac{d^2}{dx^2}U_\epsilon(x) & 0
		\end{pmatrix}$$
		
		and therefore, we need to show that $\frac{d^2}{dx^2}U_\epsilon(\frac{1}{2})<0$. But we have by the Poisson equation that 
		$$\frac{d^2}{dx^2}U_\epsilon\left(\frac{1}{2}\right)=-\frac{5\sqrt{2}\pi}{16}\left(c-U\left(\frac{1}{2}\right)\right)^3.$$
		We also have the relation 
		\begin{flalign*}
			U\left(\frac{1}{2}\right)=c-\left( c^4-\frac{64}{5\sqrt{2}\pi \epsilon}U\left( \frac{\epsilon}{2} \right)  \right)^{\frac{1}{4}}.
		\end{flalign*}
		which is also shown in \cite{Twisting} as a byproduct of integrating the Poisson equation. Then, we have
		\begin{flalign*}
			\frac{d^2}{dx^2}U_\epsilon\left(\frac{1}{2}\right)=-\frac{5\sqrt{2}\pi}{16}\left( c_\epsilon^4-\frac{64}{5\sqrt{2}\pi\epsilon}U_\epsilon\left(\frac{\epsilon}{2}\right) \right)^\frac{3}{4}=-\frac{5\sqrt{2}\pi}{16}\left( c_\epsilon^4-\frac{64}{5\sqrt{2}\pi}h_\epsilon \right)^\frac{3}{4}
		\end{flalign*}
		The right hand side is clearly nonpositive. Moreover, it is strictly negative, since otherwise, we have
		$c_\epsilon^4-\frac{64}{5\sqrt{2}\pi}h_\epsilon=0$ and this contradicts the second equation, as follows:
		\begin{flalign*}
			\frac{1}{2}&=\frac{\epsilon}{2}+\int_{\epsilon h_\epsilon}^{c_\epsilon-\left(c_\epsilon^4-\frac{64h_\epsilon}{5\sqrt{2}\pi} \right)^{\frac{1}{4}}}\frac{dw}{\sqrt{\frac{5\sqrt{2}\pi}{32}((c_\epsilon-w)^4-c_\epsilon^4)+2h_\epsilon}}\\
			&=\frac{\epsilon}{2}+\int_{\epsilon h_\epsilon}^{c_\epsilon} \left(\frac{32}{5\sqrt{2}\pi}\right)^\frac{1}{2}\frac{1}{(c_\epsilon-w)^2} dw =\infty
		\end{flalign*}
	\end{proof}
	
	\subsection{Stability of the density in $L^1$}
	To show the gradient growth, we need the stability of the density, which is a consequence of the result below:
	\begin{proposition}
		\label{L^1}
		Let $f_*$ be the steady state constructed above for some $\epsilon\in(0, \epsilon_0)$ and let $f$ be another solution for the Vlasov-Poisson equation with background charge density $\rho_\epsilon$. Then, the following estimate is true:
		\begin{flalign*}
			\left\Vert \int f(t, \cdot, v)dv-\int f_*(\cdot, v)dv \right\Vert_{L^1(\mathbb{T})} &\leq C\sqrt{v_0} \left( \int_{\mathbb{T}\times\mathbb{R}} (1+v^2)\vert f(0, x, v)-f_*(x,v) \vert dvdx \right)^\frac{1}{2} \\& \quad+ \frac{4}{v_0^2}H(f_*) +\frac{4}{v_0^2} \int_{\mathbb{T}\times\mathbb{R}} (1+v^2)\vert f(0, x, v)-f_*(x,v) \vert dvdx
		\end{flalign*}
		where $C$ is a constant depending only on $\Vert f(t)\Vert_{L^\infty_{x, v}}=\Vert f(0)\Vert_{L^\infty_{x, v}}$ and $v_0>0$ is arbitrary.
	\end{proposition}
	Before proving the proposition above, we need some result for the $L^2$ stability of $f_*$, which is also proven in \cite{Twisting}.
	\begin{proposition}
		\label{L^2}
		Let $f_*$ and $f$ be as in Proposition \ref{L^1}. Then, there exists $C>0$ depending only on $f_*$ and $\Vert f(0)\Vert_{L^\infty}$ such that the following holds:
		\begin{flalign*}
			\Vert f(t)-f_*\Vert^2_{L^2(\mathbb{T}\times\mathbb{R})} \leq C\int_{\mathbb{T}\times\mathbb{R}}(1+v^2)\vert f(0)-f_*\vert dvdx
		\end{flalign*}
	\end{proposition}
	\begin{proof}[Proof of Proposition \ref{L^1}]
		We proceed as follows:
		\begin{flalign*}
			\left\Vert \int_\mathbb{R} f(t, \cdot, v)dv-\int_\mathbb{R} f_*(\cdot, v)dv \right\Vert_{L^1(\mathbb{T})}&\leq \int_{\mathbb{T}\times\mathbb{R}} \vert f(t, x, v) -f_*(x, v)\vert dxdv
			\\&\leq \int_{\vert v\vert\leq v_0} \int_{\mathbb{T}} \vert f(t, x, v)-f_*(x, v)\vert dxdv + \int_{\vert v\vert\geq v_0} \int_{\mathbb{T}} \vert f(t, x, v)-f_*(x, v)\vert dxdv \\
			&\leq \sqrt{2v_0} \Vert f(t)-f_*\Vert_{L^2(\mathbb{T}\times\mathbb{R})} + \frac{2}{v_0^2} \int_{\vert v\vert\geq v_0}\int_{\mathbb{T}}\frac{v^2}{2}(f(t, x, v)+f_*(x, v))dxdv \\&
			\leq  \sqrt{2v_0} \Vert f(t)-f_*\Vert_{L^2(\mathbb{T}\times\mathbb{R})} + \frac{2}{v_0^2} (H(f(t))+H(f_*)) \\
			&= \sqrt{2v_0} \Vert f(t)-f_*\Vert_{L^2(\mathbb{T}\times\mathbb{R})} + \frac{2}{v_0^2} (H(f(0))+H(f_*))
		\end{flalign*}
		where we have used Cauchy-Schwarz inequality in the third line and the fact that the total energy is conserved in $t$ in the last line.  Using Proposition \ref{L^2}, we have
		\begin{flalign*}
			&\leq C\sqrt{v_0} \left( \int_{\mathbb{T}\times\mathbb{R}}(1+v^2)\vert f(0)-f_*\vert dvdx \right)^\frac{1}{2} +\frac{2}{v_0^2}(H(f(0))+H(f_*))\\
			&\leq C\sqrt{v_0} \left( \int_{\mathbb{T}\times\mathbb{R}}(1+v^2)\vert f(0)-f_*\vert dvdx \right)^\frac{1}{2} +\frac{2}{v_0^2}(2H(f_*)+\vert H(f(0))-H(f_*)\vert)\\
			&\leq C\sqrt{v_0} \left( \int_{\mathbb{T}\times\mathbb{R}}(1+v^2)\vert f(0)-f_*\vert dvdx \right)^\frac{1}{2} \\&\quad+\frac{2}{v_0^2}\left(2H(f_*)+\left\vert \int_{\mathbb{T}\times\mathbb{R}} \frac{v^2}{2}(f(0, x, v)-f_*(x, v))dvdx +\frac{1}{2}\int_{\mathbb{T}} (E(0, x)^2-E_*(x)^2)dx  \right\vert\right)\\
			&\leq C\sqrt{v_0} \left( \int_{\mathbb{T}\times\mathbb{R}}(1+v^2)\vert f(0)-f_*\vert dvdx \right)^\frac{1}{2} \\&\quad+\frac{2}{v_0^2}\left(2H(f_*)+\int_{\mathbb{T}\times\mathbb{R}} \frac{v^2}{2}\vert f(0, x, v)-f_*(x, v)\vert dvdx +\frac{1}{2}\int_{\mathbb{T}} \vert E(0, x)+E_*(x)\vert \vert E(x, 0)-E_*(x) \vert dx  \right)
		\end{flalign*}
		Using the formula 
		\begin{flalign*}
			E(x)=\int_{0}^{1}K(x, y)\left[\rho_\epsilon(y)-\int_{-\infty}^{\infty}f(y, v)dv\right]dy,
		\end{flalign*}
		we obtain the following bound for $E$:
		\begin{flalign*}
			\vert E(x)\vert &\leq \int_0^1 \vert K(x,y)\vert \left\vert \rho_\epsilon(y)+\int_\mathbb{R} f(y, v)dv \right\vert dy \leq \int_0^1 \left(\rho_\epsilon(y)+\int_\mathbb{R} f(y, v)dv\right)dy=2.
		\end{flalign*}
		Moreover, we have 
		\begin{flalign*}
			\vert E(0, x)-E_*(x)\vert \leq \int_0^1 \vert K(x, y)\vert \int_\mathbb{R} \vert f(0, y, v)-f_*(y, v)\vert dvdy \leq \int_{\mathbb{T}\times\mathbb{R}} \vert f(0, y, v)-f_*(y, v)\vert dvdy
		\end{flalign*}
		Therefore, we have 
		\begin{flalign*}
			&C\sqrt{v_0} \left( \int_{\mathbb{T}\times\mathbb{R}}(1+v^2)\vert f(0, x, v)-f_*(x, v)\vert dvdx \right)^\frac{1}{2} \\&\quad+\frac{2}{v_0^2}\left(2H(f_*)+\int_{\mathbb{T}\times\mathbb{R}} \frac{v^2}{2}\vert f(0, x, v)-f_*(x, v)\vert dvdx +\frac{1}{2}\int_{\mathbb{T}} \vert E(0, x)+E_*(x)\vert\vert E(x, 0)-E_*(x) \vert dx  \right)\\
			&\leq C\sqrt{v_0} \left( \int_{\mathbb{T}\times\mathbb{R}}(1+v^2)\vert f(0, x, v)-f_*(x, v)\vert dvdx \right)^\frac{1}{2} \\&\quad+\frac{2}{v_0^2}\left(2H(f_*)+\int_{\mathbb{T}\times\mathbb{R}} \frac{v^2}{2}\vert f(0, x, v)-f_*(x, v)\vert dvdx +2\int_{\mathbb{T}} \vert E(x, 0)-E_*(x) \vert dx  \right)\\
			&\leq C\sqrt{v_0} \left( \int_{\mathbb{T}\times\mathbb{R}}(1+v^2)\vert f(0, x, v)-f_*(x, v)\vert dvdx \right)^\frac{1}{2} \\&\quad+\frac{4}{v_0^2} H(f_*) +\frac{2}{v_0^2} \left(\int_{\mathbb{T}\times\mathbb{R}} \frac{v^2}{2}\vert f(0, x, v)-f_*(x, v)\vert dvdx +2\int_{\mathbb{T}\times\mathbb{R}} \vert f(0, x, v)-f_*(x, v) \vert dvdx  \right)\\
			&\leq C\sqrt{v_0} \left( \int_{\mathbb{T}\times\mathbb{R}} (1+v^2)\vert f(0, x, v)-f_*(x,v) \vert dvdx \right)^\frac{1}{2} \\& \quad+ \frac{4}{v_0^2}H(f_*) +\frac{4}{v_0^2} \int_{\mathbb{T}\times\mathbb{R}} (1+v^2)\vert f(0, x, v)-f_*(x,v) \vert dvdx
		\end{flalign*}
		and the proof is complete.
	\end{proof}
	Proposition \ref{L^1} implies the stability result as follows. For any $\varepsilon>0$ and $M>0$, there exists $\delta>0$ such that for all solutions $f(t, x, v)$ satisfying $$\int_{\mathbb{T}\times\mathbb{R}} (1+v^2)\vert f(0, x, v)-f_*(x,v) \vert dvdx<\delta\qquad \text{and} \qquad \Vert f(0)\Vert_{L^\infty}\leq M, $$
	we have that $\left\Vert \int f(t, \cdot, v)dv-\int f_*(\cdot, v)dv \right\Vert_{L^1(\mathbb{T})}<\varepsilon$. Indeed, for such $\varepsilon$, we can pick $v_0$ large so that 
	\begin{flalign*}
		\frac{4}{v_0^2}H(f_*)\leq \frac{\varepsilon}{2}
	\end{flalign*}
	holds, and pick $\delta>0$ small enough so that
	\begin{flalign*}
		\leq C\sqrt{v_0\delta} +\frac{4}{v_0^2} \delta\leq\frac{\varepsilon}{2}.
	\end{flalign*}
	holds, and we see that such $\delta$ works.
	\\
	Since we have $\frac{d}{dx}E(x)=\rho_\epsilon(x)-\int_{\mathbb{R}}f(t, x, v)dv$ we see that 
	\begin{flalign*}
		\vert E(x, t)-E_*(x) \vert \leq \varepsilon.
	\end{flalign*}
	Therefore, we obtain the $L^\infty$-stability of the characteristic velocity, which plays a crucial role in showing the superlinear growth for the gradient.
	\subsection{Proof for the superlinear growth}
	For the steady state $f_*$ constructed before, we see that the characteristic ODE can be expanded around $(x, v)=(1/2, 0)$ as follows:
	\begin{flalign*}
		\frac{dx}{dt}=v,\qquad \frac{dv}{dt}=A^2\left(x-\frac{1}{2}\right)+g\left(x-\frac{1}{2}\right)
	\end{flalign*}
	where $A>0$ and $g(x)=o(x)$. It follows that there exists $\gamma>0$ such that we have $\vert g(x)\vert \leq 0.01A\vert x\vert $ for all $\vert x\vert\leq 4\gamma$. Also, by the $L^\infty$ stability of $E$ obtained in the previous section, we see that if $$\int_{\mathbb{T}\times\mathbb{R}} (1+v^2)\vert f(0, x, v)-f_*(x,v) \vert dvdx<\delta \qquad \text{and} \qquad \Vert f(0)\Vert_{L^\infty}\leq 2\Vert f_*\Vert_{L^\infty},$$ then $\Vert E(x, t)-E_*(x) \Vert\leq 0.0001\gamma A$ for all $t\geq 0$. let us assume that this is the case. Note that the characteristic ODE for $f$ is 
	\begin{flalign*}
		\frac{dx}{dt}=v,\qquad \frac{dv}{dt}=A^2\left(x-\frac{1}{2}\right)+g\left(x-\frac{1}{2}\right)+(E(x, t)-E_*(x))
	\end{flalign*}
	Now, if we make change of coordinates
	\begin{flalign}
		\label{coordinate}
		\begin{cases}
			\xi= \frac{A(x-\frac{1}{2})+v}{\sqrt{1+A^2}} \\
			\eta= \frac{A(x-\frac{1}{2})-v}{\sqrt{1+A^2}}.
		\end{cases}
	\end{flalign}
	Then, we see that the characteristic ODE can be rewritten as
	\begin{flalign}
		\label{char}
		\frac{d\xi}{dt} = A\xi+g_1(\xi, \eta)+h_1(\xi, \eta, t), \qquad \text{and}\qquad \frac{d\eta}{dt} = -A\eta+g_2(\xi, \eta)+h_2(\xi, \eta, t)
	\end{flalign}
	where $\vert g_{(1, 2)}(\xi, \eta)\vert \leq 0.02 A\sqrt{\xi^2+\eta^2}$ and $\vert h_{(1, 2)}(\xi, \eta, t)\vert \leq 0.0002\gamma A$.
	We need a lemma that gives some information about such flow.
	\begin{lemma}
		\label{flow}
	Consider the system (\ref{char}). Then, for $\tau\leq \frac{1}{10A}$, the followings are true:
	\begin{enumerate}
		\item If  $\vert \xi(0)\vert \leq \gamma$ and $\vert \eta(0)\vert<0.1\gamma$, then $\vert \eta(\tau)\vert <0.1\gamma$. 
		\item If  $(1-\frac{A\tau}{2})\gamma < \vert \xi(0)\vert < \gamma$ and $\vert \eta(0)\vert<0.1\gamma$, then $\vert \xi(\tau)\vert >\gamma$. 
	\end{enumerate}
	\end{lemma}
	\begin{proof}
		We show that if $\sqrt{\xi(0)^2+\eta(0)^2} <2\gamma$, then $\sqrt{\xi(t)^2+\eta(t)^2}\leq 3\gamma$ for $t\leq \frac{1}{10A}$. Indeed, we have
		\begin{flalign*}
			\frac{d}{dt}\left( \xi(t)^2+\eta(t)^2 \right)&=2A(\xi(t)^2-\eta(t)^2)+2\xi(t)(g_1(\xi, \eta)+h_1(\xi, \eta, t))+2\eta(t)(g_2(\xi, \eta)+h_2(\xi, \eta, t)) \\&\leq 2A(\xi(t)^2+\eta(t)^2)+0.04A(\xi(t)^2+\eta(t)^2)+0.0004\sqrt{\xi(t)^2+\eta(t)^2}\gamma A.
		\end{flalign*}
		which implies 
		\begin{flalign*}
			\frac{d}{dt}\sqrt{ \xi(t)^2+\eta(t)^2}\leq 1.02A\sqrt{\xi(t)^2+\eta(t)^2}+0.0002\gamma A.
		\end{flalign*}
		Using Gronwall inequality, we see that 
		\begin{flalign*}
			\sqrt{ \xi(t)^2+\eta(t)^2}\leq \sqrt{ \xi(0)^2+\eta(0)^2}e^{1.02At}+\frac{0.0002\gamma A}{1.02A}(e^{1.02At}-1)<2\gamma e^{0.102}+\frac{0.0002\gamma}{1.02}(e^{0.102}-1) <3\gamma.
		\end{flalign*}
		Suppose now that $\vert \xi(0)\vert \leq \gamma$ and $\vert \eta(0)\vert<0.1\gamma$. Then, since
		\begin{flalign*}
			\frac{d\eta}{dt}=-A\eta+g_2(\xi, \eta)+h_2(\xi, \eta, t)
		\end{flalign*}
		and $$\vert g_2(\xi, \eta)+h_2(\xi, \eta, t) \vert\leq 0.0602\gamma A, $$
		we see that 
		\begin{flalign*}
			\vert \eta(\tau)\vert \leq \vert \eta(0)\vert e^{-A\tau}+(1-e^{-A\tau})0.0602\gamma \leq 0.1\gamma e^{-0.1}+(1-e^{-0.1})0.0602\gamma <0.1\gamma.
		\end{flalign*}
		which proves the first statement.\\
		
		Suppose that $(1-\frac{A\tau}{2})\gamma < \vert \xi(0)\vert < \gamma$ and $\vert \eta(0)\vert<0.1\gamma$.
		Since
		\begin{flalign*}
			\frac{d\xi}{dt}=A\xi+g_1(\xi, \eta)+h_1(\xi, \eta, t)
		\end{flalign*}
		and $$\vert g_1(\xi, \eta)+h_1(\xi, \eta, t) \vert\leq 0.0602\gamma A, $$
		We have
		\begin{flalign*}
			\vert \xi(\tau)\vert &\geq \vert \xi(0)\vert e^{A\tau}-\frac{e^{A\tau}-1}{A}0.0602\gamma A\geq (1-\frac{A\tau}{2})\gamma e^{A\tau}-(e^{A\tau}-1)0.0602\gamma \\&\geq (1+A\tau)(1-\frac{A\tau}{2})\gamma-0.0602\gamma A\tau=\gamma+(0.5-0.0602)\gamma A\tau-\frac{(A\tau)^2}{2}\gamma\\&
			\geq \gamma+(0.5-0.0602-0.05)\gamma A \tau >\gamma.
		\end{flalign*}
		where we have used the fact that $e^{A\tau}>1+A\tau$. This proves the second statement.
	\end{proof}
	Now, we construct the initial condition that exhibits the desired gradient growth. We contruct $f_0$ as follows:
	\begin{enumerate}
		\item $f_0$ is symmetric about $O_1=(1/2, 0)$, i.e. $f_0(1-x, -v)=f_0(x, v)$.
		\item $f_0$ is nonnegative and $\Vert f_0\Vert_{L^\infty}\leq 2\Vert f_*\Vert_{L^\infty}$.
		\item Under the coordinate transform (\ref{coordinate}), we have that $f_0=2\Vert f_*\Vert_{L^\infty}$ on the segment $\{ \eta=0, \vert \xi\vert\leq  0.992\gamma\}$ and $f_0=\Vert f_*\Vert_{L^\infty}$ on the ellipse
		\begin{flalign*}
			\left(\frac{\xi}{0.996\gamma}\right)^2+\left(\frac{\eta}{0.05\gamma}\right)^2=1.
		\end{flalign*}
		\item $f_0=f_*$ outside $U_\gamma:=\{x^2+v^2\leq \gamma^2\}$ and $\Pi_\gamma = \{ \vert\xi\vert\leq \gamma,\; \vert \eta\vert \leq 0.1\gamma \}$, where $(\xi, \eta)$ are as in $(\ref{coordinate})$.
		\item $f_0(x, v)=f_*(x, v)-\chi_\gamma(x, v)$ inside $U_\gamma$, where $\chi_\gamma\in C^\infty_c(U_\gamma)$ is positive, spherically symmetric, and such that $\int_{\mathbb{T}\times\mathbb{R}}f_0(x, v)dvdx=1$.
		\item $\int_{\mathbb{T}\times\mathbb{R}} (1+v^2)\vert f_0(x, v)-f_*(x,v) \vert dvdx<\delta$
		\item $f_0$ is smooth everywhere.
	\end{enumerate}
	One can show that such $f_0$ exists. Now, let $f$ be the solution for the Vlasov-Poisson equation with initial condition $f_0$, and $\Phi(t_1, t_2)$ be the characteristic flow from time $t_1$ to $t_2$. Then, by the above argument, we see that $\Vert E(x, t)-E_*(x) \Vert\leq 0.0001\gamma A$ for all $t\geq 0$.
	
		Having constructed the initial data, we now prove the gradient growth for such solution. Let $\{t_n\}$ be sequence of points such that $t_n\in[\frac{2n-1}{50A}, \frac{n}{25A}]$ and $\Vert \nabla f(t_n) \Vert_{L^\infty}=\min_{t\in[\frac{2n-1}{50A}, \frac{n}{25A}]}\Vert \nabla f(t) \Vert_{L^\infty}$. Note that $t_n-t_{n-1}\in [\frac{1}{50A}, \frac{3}{50A}]$. Consider the set of points inside the $f=3$ level curve at $t=0$, which is an ellipse. Let $S_0$ be the intersection of this set and $B_o=\{\vert \xi\vert \leq \gamma, \vert\eta\vert<0.1\gamma\}$. Let $S_0^1$ be the intersection of $S_0$ and $B_i=\{\vert \xi\vert \leq 0.99\gamma, \vert\eta\vert<0.1\gamma\}$, and let $S_0^2=S_0-S_0^1$. Consider the set $\Phi(0, t_1)S_0$, which will stay inside the tube $\{\vert \eta\vert <0.1\gamma \}$ by Lemma \ref{flow}. Let $S_{1}$ be the simply connected component of the set $\Phi(0, t_1)S_0\cap B_o$ containing the point $(x, v)=(1/2, 0)$. Then, we observe that the followings are true:
	\begin{enumerate}
		\item $\partial S_1$ consists of the part of the level curve $\{f(t_1, x, v)=3\}$ and the part of the vertical segments $\xi=\pm\gamma$.
		\item $\vert S_1 \vert \leq \vert S^1_0\vert =\vert S_0\vert -\vert S_0^2\vert $ since the flow carries $S_0^2$ away from $B_o$ by Lemma \ref{flow}.
		\item $S_1$ contains the part of the level curve $f(t_1)=4$ which is symmetric with respect to $O_1$ and connects the following points: the origin and $P_{\pm}^1$, where the $\xi$-coordinate of $P_{\pm}^1$ is $\pm\gamma$. This follows from the fact that edge points are carried away from $B_o$ at $t=t_1$.
		\item For the decomposition $S_1=S_1^1\cup S_1^2$ analogous to the decomposition $S_0=S_0^1\cup S_0^2$ the set $S_1^2$ has positive area.  
	\end{enumerate}
	We inductively define $S_n$ for all times $t=t_n$. All the properties above will still hold. In particular, we have
	\begin{flalign*}
		\vert S_{n+1}\vert \leq \vert S_n\vert -\vert S_n^2\vert
	\end{flalign*}
	and therefore
	\begin{flalign*}
		\sum_{n=0}^\infty \vert S_n^2\vert <\infty.
	\end{flalign*}
	Now, we estimate $\vert S_n^2\vert $. Note that $S_n^2$ is symmetric with respect to $O_1$ and therefore we only consider the right part. The sides $\xi=0.99\gamma$ and $\xi=\gamma$ contains the points of the level curve $f(t_n)=4$. By mean value theorem, we therefore have that
	\begin{flalign*}
		\vert S_n^2\vert \gtrsim \gamma^2\Vert \nabla f(t_n) \Vert_{L^\infty}^{-1}.
	\end{flalign*}
	We finally use the the following lemma, which can be found in \cite{growth_EE}.
	\begin{lemma}[Lemma 2.3 in \cite{growth_EE}]
		If $a_j>0$ and $$\sum_{j=0}^\infty a_j<\infty, $$
		then
		\begin{flalign*}
			\frac{1}{N^2}\sum_{j=0}^N a_j^{-1}\to\infty.
		\end{flalign*}
	\end{lemma}
	By the above lemma, we see that
	\begin{flalign*}
		\lim_{N\to\infty} \frac{1}{N^2}\sum_{j=0}^N \Vert \nabla f(t_j)\Vert_{L^\infty}=\infty.
	\end{flalign*}
	Since
	\begin{flalign*}
		\frac{(25A)^2}{N^2}\int_{0}^{\frac{N}{25A}} \Vert\nabla f(t)\Vert_{L^\infty}dt >\frac{(25A)^2}{N^2}\sum_{j=1}^N \int_{\frac{2j-1}{50A}}^{\frac{j}{25A}} \Vert\nabla f(t)\Vert_{L^\infty}dt \geq \frac{25A}{2N^2}\sum_{j=1}^N \Vert \nabla f(t_j)\Vert_{L^\infty}
	\end{flalign*}
	we have
	\begin{flalign*}
		\lim_{T\to\infty}\frac{1}{T^2}\int_0^T \Vert \nabla f(t)\Vert_{L^\infty} dt=\infty
	\end{flalign*}
	\section*{Acknowledgement}
	
	The author would like to express his gratitude and thanks to Professor In-Jee Jeong for his guidance throughout this project. 
    
\end{document}